\documentclass[12pt,twoside]{amsart}
\usepackage{a4wide}
\usepackage{microtype}

\usepackage{mathrsfs}
\usepackage{mathtools}
\usepackage{amssymb}
\usepackage{verbatim}
\usepackage{amsmath}
\usepackage{comment}
\usepackage{amsthm,thmtools,xcolor}
\usepackage[colorlinks,linkcolor=blue,citecolor=blue, pdfstartview=FitH]
{hyperref}
\usepackage{backref}
\usepackage{hyperref}

\usepackage{bm}
\usepackage[latin1]{inputenc}
\usepackage[T1]{fontenc}
\usepackage{times}
\usepackage{hyperref}
\usepackage{amssymb,latexsym}

\usepackage{mathrsfs}
\usepackage{mathtools}
\usepackage{amssymb}
\usepackage{verbatim}
\usepackage{amsmath}
\usepackage{comment}
\usepackage{amsthm,thmtools,xcolor}
\usepackage[colorlinks,linkcolor=blue,citecolor=blue, pdfstartview=FitH]
{hyperref}
\usepackage{backref}
\usepackage{bm}
\usepackage{a4wide}
\usepackage[latin1]{inputenc}
\usepackage[T1]{fontenc}
\usepackage{times}
\usepackage{hyperref}
\usepackage{amssymb,latexsym}

\usepackage{etoolbox}

\newtheorem{theorem}{Theorem}[section]

\numberwithin{equation}{section}

\newtheorem{proposition}[theorem]{Proposition}

\newtheorem{corollary}[theorem]{Corollary}
\newtheorem{lemma}[theorem]{Lemma}

\theoremstyle{definition}
\newtheorem{definition}[theorem]{Definition}

\theoremstyle{remark}

\usepackage{enumerate}

\newcommand{\dbar}{\ensuremath{\overline\partial}}

\newcommand{\R}{\ensuremath{\mathbb{R}}}

\newcommand{\abs}[1]{\lvert#1\rvert}

\begin{document}
	
	\title[Hilbert bundle approach to complex Brunn-Minkowski theory]{A Hilbert bundle approach to the sharp strong openness theorem and the Ohsawa-Takegoshi extension theorem}
	
	\author{Tai Terje Huu Nguyen}
	\address{DEPARTMENT OF MATHEMATICAL SCIENCES, NORWEGIAN UNIVERSITY OF SCIENCE AND TECHNOLOGY, NO-7491 TRONDHEIM, NORWAY (current:
	Department of Business, Strategy and Political Sciences, University of South Eastern Norway
Hasbergsvei 36, 3616 Kongsberg, Norway)}
	\email{taiterje@hotmail.com}
	\thanks{It is a pleasure to thank Bo Berndtsson, Bo-Yong Chen and Takeo Ohsawa for
		reading an early version of this paper and their valuable comments. Thanks are also given to the anonymous referees for valuable suggestions that greatly improved the paper, in particular, for finding a mistake in our first proof of Theorem 5.1 and the counterexample in the remark to Lemma 5.2.}

	\author{Xu Wang}
	\address{DEPARTMENT OF MATHEMATICAL SCIENCES, NORWEGIAN UNIVERSITY OF SCIENCE AND TECHNOLOGY, NO-7491 TRONDHEIM, NORWAY}
	\email{xu.wang@ntnu.no}
	
	\date{\today}
	
	
	
	\begin{abstract} The following paper is around parts of the first named author's thesis. We discuss (what we call) a Hilbert bundle approach to complex Brunn-Minkowski theory and obtain a general monotonicity theorem. As two applications, we prove a generalization of Guan's sharp strong openness theorem and a sharp Ohsawa-Takegoshi extension theorem. A second proof of Guan-Zhou's strong openness theorem using a Donnelly-Fefferman estimate is also given. 

	\end{abstract}
	
	\maketitle
	
	

	\section{Introduction}
	
	In this paper, we are interested in using techniques from Berndtsson's complex Brunn-Minkowski theory to discuss sharp estimates in the strong openness theorem and the Ohsawa-Takegoshi extension theorem. The Ohsawa-Takegoshi extension theorem, which first appeared in \cite{OT}, is a very powerful result, and much studied topic, in complex geometry. In  
	\cite{Ohsawa}, Ohsawa found that the sharp version of the theorem would directly give the Suita conjecture (see \cite{Suita}), on a precise Robin constant lower bound for the Bergman kernel (of a plane domain). The sharp version was finally proved by Blocki (for pseudoconvex domains in $\mathbb C^n$) and Guan-Zhou (a very general form) in \cite{Bl0, GZ}. A completely different proof was given by Berndtsson-Lempert in \cite{BL} using complex Brunn-Minkowski theory, more specifically, using the convexity theorem of Berndtsson in \cite{Bern06, Bern09}. Another application of Berndtsson's convexity theorem is a proof of the openness conjecture, given by Berndtsson in \cite{Bern15b}. In the same paper (section 5, in \cite{Bern15b}), Berndtsson also proposed a conjectural picture for the strong openness conjecture using the convexity theorem. Later, in \cite{GZ0}, Guan-Zhou used the Ohsawa-Takegoshi extenstion theorem to prove the strong openness conjecture. In \cite{Bern20}, Berndtsson proved a compact K\"ahler case of his conjecture from \cite{Bern15b}, which can be viewed as a global case of the strong openness problem. The main difficulty in proving the conjecture of Berndtsson in the more general case is to obtain an infinite dimensional version of Lemma 2.3 in \cite{Bern20}. This is still an open problem, as far as we know. In this paper, we shall instead get by by using an $L^2$ H\"ormander-type of result that we also prove. We have chosen to name the latter "the Jet selection theorem"  \footnote{Thanks to one of the referees for suggesting the name "Jet selection theorem" to us}. The Jet selection theorem (see Theorem \ref{th:L21} in section 2)  allows us to generalize Berndtsson's method in \cite{Bern20} to the non-compact case and prove a general monotonicity result, which is the main theorem (see Theorem A in section 2) in this paper. As two applications of it, we prove a generalization of Guan's sharp version \cite{G} of the strong openness theorem  (see Theorem B in section 2) and a sharp Ohsawa-Takegoshi theorem (see Theorem C in section 2). 
	
	We end our short introduction with the organization of the paper. In section 2, we present (the statements of) our Jet selection theorem and the main theorem. We also discuss in there applications of the main theorem to the sharp strong openness theorem and the sharp Ohsawa-Takegoshi theorem. In section 3, we introduce (what we call) a Hilbert bundle approach to the complex Brunn-Minkowski theory and prove our main theorem, assuming the statement of the Jet selection theorem. In section 4, we prove the Jet selection theorem in a simple setting and give the reference for the proof of the general case. The last section of this paper, section 5, is on further applications and results. We in there give a proof of the Demailly-Koll\'ar semi-continuity theorem using the Jet selection theorem, and a second proof of Guan-Zhou's strong openness theorem using a Donnelly-Fefferman version of the Jet selection theorem.

	\section{Jet selection theorem,  main theorem and applications}
	In this section, we give the statements of the Jet selection theorem (Theorem \ref{th:L21}) and the main theorem (Theorem A), and discuss applications of the latter to sharp strong openness and sharp $L^2$-holomorphic extension. We commence by giving the set-up for the statements the theorems.
	\subsection{Quasi-complete K\"ahler manifolds and the Jet selection theorem}
	 Let $X$ be a complex manifold, and denote by $C_0^\infty (X, [0,1])$ the space of smooth maps from $X$ to $[0,1]$ with compact support in $X$. We will be dealing with the following class of K\"ahler manifolds, which comes from the main result in \cite{CWW} (see also section 6.1 in \cite{wang}); the support of a function $\chi$ is denoted $\text{supp }\chi$:
	
	\begin{definition}\label{de:qck} Let $(X, \omega)$ be a K\"ahler manifold. Then we say that it is  (a) quasi-complete K\"ahler (manifold) if there exist a subfamily $\{\chi_{j}\}_{j\geq 1} \subset C_0^\infty (X, [0,1])$, and K\"ahler open neighborhoods $(X_j, \omega_j)$ of ${\rm supp} \,\chi_j$ such that:	
\begin{enumerate}[(i)]
\item $X_{j}\subset X_{j+1}$ with $X=\bigcup_{j=1}^{\infty}X_{j}$,
\item $\omega_{j}\geq \omega$ on $X_{j}$,
\item 	$$\lim_{j\to \infty}  \sup_{X}|\dbar\chi_j|_{\omega_j} =0,$$ for every compact subset $K$ of $X$, $$ \lim_{j\to \infty} \sup_{K} |\omega_j-\omega|_{\omega} =0,$$ and
\item for every compact subset $K$ of $X$, $$\chi_j|_{K}\equiv 1$$ for all large $j$.
\end{enumerate}	In the case that (i) - (iv) hold, we also call $(X_j, \omega_j, \chi_j)$ an approximation family for $(X, \omega)$.
		\end{definition}
		
\noindent \textbf{Remark.} \emph{For simplicity, we shall use the abbreviation qck for ''quasi-complete K\"ahler''. }\\
	
	Recall that a K\"ahler manifold $(X, \hat \omega)$ is said to be complete if there exists a real-valued smooth function $\rho$ on $X$ such that
		for every $c\in \mathbb R$, the sub-level set 
		$\{\rho<c\}
		$
		is relatively compact in $X$, and $
		|\dbar \rho|_{\hat\omega} \leq 1
		$ on $X$. One of our main motivations for introducing qck manifolds is the following fact:
	
	\begin{proposition}\label{pr:qck} Assume that $(X, \hat \omega)$ is complete K\"ahler. Then $(X, \omega)$ is qck for every K\"ahler form $\omega$ on $X$.
	\end{proposition}
	
	\begin{proof} One may take
		$
		X_j:=X, \ \omega_j:=\omega+(1/j) \hat \omega, \ \ \chi_j:= \chi(\rho/ j^2),  
		$
		with $\chi:\R\to \R$ smooth such that $\chi=1$ on $(-\infty, 1]$, $\chi=0$ on $[3, \infty)$, $0\leq \chi\leq 1$, and $|\chi'| \leq 1$.
	\end{proof}
	
	Another of our motivations is rooted in trying to find the right set-up for utilizing the Demailly approximation in H\"ormander's $L^2$-estimate, in which case one needs to solve the $\dbar$-equation on the complement of an analytic subset of a complete K\"ahler manifold (see \cite[page 89-91]{Nguyen}). The point is then that while the complement might not be complete K\"ahler in general (in fact, it is a conjecture that there exists a counterexample), it is always qck by the above proposition and the following result  (see \cite[page 74, Lemma 4.1]{Nguyen} for the proof):
	
	\begin{proposition}\label{pr:qck1} Assume that $(X, \omega)$ is qck. Then $(X\setminus S, \omega)$ is qck for every analytic subset $S$ of $X$.
	\end{proposition}
	
	Proposition \ref{pr:qck1} allows us to prove the Jet selection theorem mentioned in the introduction (Theorem \ref{th:L21}), whose statement we are now ready to give. We use the following conventions: $K_{X}$ denotes the canonical line bundle of $X$, and a line bundle $(L,e^{-\phi})$ (over $X$) being pseudoeffective is to mean that the local weight $\phi$ is plurisubharmonic. Then the Jet selection theorem is the following:
	
	\begin{theorem}[The Jet selection theorem]\label{th:L21} Let $(L, e^{-\phi})$ be a pseudoeffective  line bundle over an $n$-dimensional qck manifold $X$, and introduce a family of norms $||\cdot||_{t}$ on the space $H^0(X,\mathcal{O}(K_X+L))$ indexed by $t\in \mathbb{R}$, by letting for $u\in H^0(X,\mathcal{O}(K_X+L))$, 
		\begin{equation}\label{eq:tNorm}
			||u||_t^2:=\int_X i^{n^2}  u\wedge \overline{  u}  \, e^{-\phi -\lambda \max\{G-t, 0\}},
		\end{equation}
		where $G\leq 0$ is a function on $X$ such that
		$\phi+\lambda G$ is plurisubharmonic on $X$ for some constant $\lambda>0$. 
		Fix $t<0$ and $0 \leq \alpha <\lambda$, and let $F\in H^0(X, \mathcal O(K_X+L))$  with
		$
		||F||_t^2 <\infty.
		$
		Then there exists $\tilde F\in H^0(X, \mathcal O(K_X+L))$ such that
		\begin{equation}\label{eq:h11-new}
			||\tilde F-F||^2_{\phi+\alpha G} \leq  \frac{9 \lambda(1-t)^{\lambda -\alpha+1} e^{-2\alpha t}}{\lambda-\alpha}  \,   ||F||_t^2 + 2 e^{-\alpha t} ||F||^2_\phi
		\end{equation}
		and
		\begin{equation}\label{eq:h12-new}
			||\tilde F||^2_\phi \leq \frac{4\lambda(1-t)^{\lambda -\alpha+1} e^{-\alpha t}}{\lambda-\alpha}\, ||F||_t^2,
		\end{equation}
		where $||u||_\psi^2:=\displaystyle \int_X i^{n^2} u\wedge \bar u \, e^{-\psi} $.
	\end{theorem}
	
	Our main observation in this paper is the following: One may use the Jet selection theorem (Theorem \ref{th:L21}) to replace the infinite dimensional version of Lemma 2.3 in \cite{Bern20} and generalize the theory in \cite{Bern20} to a large class of non-compact manifolds (see the main theorem, Theorem A, below). 
	
	\subsection{The main theorem, a generalized sharp strong openness theorem, and a sharp Ohsawa-Takegoshi extension theorem} 
	We next give the statement of the main theorem in this paper, and discuss applications of it to a generalization of the sharp strong openness theorem and a sharp Ohsawa-Takegoshi extension theorem. For a locally integrable function $\psi$, we will denote by $\mathcal I(\psi)$ the sheaf of germs of holomorphic functions $f$ such that $|f|^2e^{-\psi}$ is locally integrable. Our main theorem is the following: 
	
	\medskip
	\noindent
	\textbf{Theorem A (See page 88 of \cite{Nguyen}).} \emph{Let $(L, e^{-\phi})$ be a pseudoeffective line bundle over a qck manifold $(X, \omega)$, and let $G\leq 0$ be a function on $X$ such that $\phi+\lambda G$ is plurisubharmonic on $X$ for some constant $\lambda>0$.
		Fix $F\in H^0(X, \mathcal O(K_X+L))$ with $||F||^2_\phi<\infty$. Then for every $0\leq \alpha<\lambda$, 
		$$
		\mathbb{R}\ni t \mapsto e^{-\alpha t} \inf\left\lbrace ||\tilde F||_t^2 :   \tilde F-F \in H^0(X, \mathcal O(K_X+L)\otimes\mathcal I(\phi+\alpha G)) \right\rbrace   
		$$
		is decreasing, where the norms $||\cdot||_t$ are defined as in \eqref{eq:tNorm}, by} 
		
		\begin{align*}
			||u||_t^2:=\int_X i^{n^2}  u\wedge \overline{  u}  \, e^{-\phi -\lambda \max\{G-t, 0\}}.
		\end{align*}

	\medskip
	
	In the case that $X$ is compact, the above theorem follows from the main theorem in \cite{Bern20}. Our motivation is to use the Jet selection theorem (Theorem \ref{th:L21}) to prove the non-compact case. The proof of Theorem A is given in the next section, and the proof of the Jet selection theorem (Theorem \ref{th:L21}) is given (in a simple setting) in section 4.\\
	
We now discuss applications, starting with strong openness.  When $X$ is the unit ball in $\mathbb{C}^{n}$, Theorem A gives the following effective strong openness theorem: 
	
	\medskip
	\noindent
	\textbf{Theorem B.} \emph{Let $F$ be a holomorphic function on the unit ball $\mathbb B$ in $\mathbb C^n$, $\phi$ a plurisubharmonic function on $\mathbb B$, and $G\leq 0$ a function on $\mathbb B$ such that $\phi+ \lambda G$ is  plurisubharmonic
		for some positive constant $\lambda >0$.
		Suppose that 
		$$
		\int_\mathbb B |F|^2  \,e^{-\phi-\beta G}<\infty
		$$ for some constant $0<\beta<\lambda $, where integration is with respect to the Lebesgue measure. Then $|F|^2  \,e^{-\phi-\alpha G}$ is integrable near the origin for some $\alpha>\beta$. Moreover, letting
		$$
		\alpha_0:=\sup \{\alpha\geq 0: F_0 \in \mathcal I(\phi+\alpha G)_0\},
		$$
		then if $\alpha_0 <\lambda$, we have
		\begin{equation}\label{eq:so-sharp1-new}
			I_0:=\inf \left\lbrace\int_{\mathbb B} |\tilde F|^2e^{-\phi}:  (\tilde F-F)_0 \in \bigcup_{\alpha_0<\alpha\leq\lambda}  \mathcal I(\phi+\alpha G)_0 \right\rbrace >0,
		\end{equation}
		and
		\begin{equation}\label{eq:so-sharp2-new}
			\alpha_0 -\beta \geq  \frac{I_0}{\frac{\lambda}{(\lambda-\beta)\beta} \int_\mathbb B |F|^2  \,e^{-\phi-\beta G}- \frac1{\beta}\int_\mathbb B |F|^2 e^{-\phi} }.
	\end{equation}}
	
	\begin{proof} It suffices to prove \eqref{eq:so-sharp1-new} and \eqref{eq:so-sharp2-new}. 
		
		\smallskip
		
		\emph{Step 1: proof of \eqref{eq:so-sharp1-new}.} Nadel \cite{Nad} proved that each $\mathcal I(\phi+\alpha G)$ is a coherent ideal of the sheaf $\mathcal O_\mathbb B$ of germs of local holomorphic functions on $\mathbb B$. Since $G\leq 0$, we get that $\mathcal I(\phi+\alpha G)$ is decreasing with respect to $\alpha$. Thus the strong Noetherian property of coherent analytic sheaves (see \cite[page 90]{Demailly12}) gives
		\begin{equation}\label{eq:Noetherian0}
			\bigcup_{\alpha_0<\alpha\leq\lambda} \,\mathcal I(\phi+\alpha G)_0= \mathcal I(\phi+\bm \alpha G)_0 \ \text{for some $\alpha_0< \bm \alpha \leq\lambda$.}
		\end{equation}
		Now we have
		\begin{equation}\label{eq:so00}
			I_0=\inf \left\lbrace\int_{\mathbb B} |\tilde F|^2e^{-\phi}:  (\tilde F-F)_0 \in \mathcal I(\phi+\bm \alpha G)_0  \right\rbrace 
		\end{equation}
		for some $\alpha_0<\bm \alpha \leq \lambda$. If $I_0=0$ then we can find holomorphic functions $F_j$ on $\mathbb B$ with
		$$
		\int_\mathbb B |F_j|^2 e^{-\phi} \leq 2^{-j} , \ (F_j-F)_0 \in \mathcal I(\phi+\bm \alpha G)_0.
		$$
		Thus $F_j \to 0$ locally uniformly on $\mathbb B$ and Cartan's closedness of analytic ideals gives $F \in \mathcal I(\phi+\bm \alpha G)_0$, which contradicts that $\alpha_0<\bm \alpha$. Hence $I_0>0$.
		
		\smallskip
		
		\emph{Step 2: proof of \eqref{eq:so-sharp2-new}.}  Using the following calculus identity from \cite{Bern15b} (which also follows by a direct computation):
		\begin{equation}\label{eq:so3}
			\int_{-\infty}^0 e^{-\lambda\max\{s-t, 0\}} e^{-\beta t} dt= \frac{\lambda}{(\lambda-\beta)\beta}  e^{-\beta s}-\frac1\beta, \ \ \forall  \ \lambda>\beta, \ \ s\leq 0,
		\end{equation}
		we get
		\begin{equation}\label{eq:so4}
			\int_{-\infty}^0 ||F||_t^2 e^{-\beta t} dt = \frac{\lambda}{(\lambda-\beta)\beta} \int_\mathbb B |F|^2  \,e^{-\phi-\beta G}- \frac1{\beta}\int_\mathbb B |F|^2e^{-\phi}, 
		\end{equation}
		where
		\begin{equation}\label{eq:tnorm}
			||F||_t^2:=\int_\mathbb B |F|^2 e^{-\phi-\lambda\max\{G-t, 0\}}.
		\end{equation}
		Our main theorem (Theorem A) with $X:=\mathbb B$ and $\alpha_0<\alpha<\bm\alpha$ gives
		$$
		e^{-\alpha t}A(t) \geq A(0), \ \forall \ t<0,
		$$
		where
		$$
		A(t):=\inf\left\lbrace ||\tilde F||_t^2 :   \tilde F-F \in H^0(X, \mathcal O(K_X+L)\otimes\mathcal I(\phi+\alpha G)) \right\rbrace.
		$$
		Thus we have
		$$
		||F||_t^2 \geq A(t) \geq e^{\alpha t} A(0) \geq e^{\alpha t} I_0, \ \forall \ \alpha_0<\alpha<\bm\alpha.
		$$ 
		Letting $\alpha \to \alpha_0$, we get $||F||_t^2 \geq e^{\alpha_0 t} I_0$, which gives
		$$
		\int_{-\infty}^0 ||F||_t^2 e^{-\beta t} dt \geq \left(\int_{-\infty}^0 e^{\alpha_0 t} e^{-\beta t} dt \right) I_0 = \frac{I_0}{\alpha_0-\beta}.
		$$
		Hence we have
		$$
		\alpha_0-\beta \geq \frac{I_0}{\int_{-\infty}^0 ||F||_t^2 e^{-\beta t} dt}.
		$$
		Equation \eqref{eq:so-sharp2-new} now follows from \eqref{eq:so4}.
	\end{proof}

	\noindent
	\textbf{Remark.} \emph{In the case that $G$ is plurisubharmonic and $\phi=0$, letting $\lambda\to \infty$, we obtain
		$$
		\alpha_0 -\beta \geq \frac{\beta \,I_0}{\int_\mathbb B |F|^2  \,e^{-\beta G}- \int_\mathbb B |F|^2  } \geq \frac{\beta \,I_0}{\int_\mathbb B |F|^2  \,e^{-\beta G}- I_0  },
		$$
		which directly gives Guan's sharp estimate
		\begin{equation}\label{eq:so-sharp2}
			\frac{\alpha_0 - \beta}{\alpha_0} \geq  \frac{I_0}{ \int_\mathbb B |F|^2  \,e^{-\beta G} }.
		\end{equation}
		}
	The non-effective version is called the strong openness conjecture of Demailly (see \cite[Remark 15.2.2]{Dem00}). It was first proved in general by Guan-Zhou in \cite{GZ0}, while the $n=2$ case was first proved in \cite{FJ, JM} (see also \cite{Hiep, Lempert, GZ1} for other proofs). In the case that $F=1$, the strong openness conjecture is also called the openness conjecture, and was first proved by Berndtsson in \cite{Bern15b}. Recent developments around the strong openness theory include \cite{BG, BG1, G1, GM, GY}. In \cite{BG}, Bao-Guan improved \eqref{eq:so-sharp2} by replacing $I_0$ with a Bergman kernel term. The effective $L^p$ version of the strong openness property is proved in \cite{BG1,GY}. The proof in \cite{GY} depends on the concavity property of minimal $L^2$ integrals over the sub-level sets of plurisubharmonic functions in \cite{G1, GM}. In Theorem B, we look at a more general quasi-plurisubharmonic $G$. As far as we know, the concavity property associated with the sub-level sets of quasi-plurisubharmonic functions is still unknown. Our proof is also different from \cite{BG,GY} in the sense that we only use the Berndtsson convexity theorem for product domains and the classical H\"ormander $L^2$- estimate. That is, our arguments involve neither of the non-product version of Berndtsson's theorem, the Ohsawa-Takegoshi theorem, and the ODE-trick for twisted Bochner-Kodaira-Nakano formula.
	\medskip
	
	Next, we turn to $L^2$-holomorphic extension and our second application of Theorem A.  That is, the following sharp Ohsawa-Takegoshi estimate; the reader may refer to \cite{B96, B05, BCP, BL, Bl0, Bl1, Chan, CC, CDM, Chen, CWW, Dem18, DHP, GZ, Hosono0, K1, K2, KS, Le2, Ma, MV1, MV2, MV3, O2, O3, Ohsawa, Siu, wang21, ZZ} for many other interesting generalizations and applications of the Ohsawa-Takegoshi theorem \cite{OT}:
	
	\medskip
	\noindent
	\textbf{Theorem C (See page 102 of \cite{Nguyen}).} \emph{Let $(L, e^{-\phi})$ be a pseudoeffective line bundle over an $n$-dimensional qck manifold $(X, \omega)$, and let $G\leq 0$ be a function on $X$ such that $\phi+\lambda G$ is plurisubharmonic on $X$ for some constant $\lambda>1$. Fix $F\in H^0(X, \mathcal O(K_X+L))$ with 
		$$||F||^2_\phi =\displaystyle \int_{X}i^{n^2}F\wedge \bar{F}e^{-\phi}<\infty.$$ Then there exists $\tilde F $ satisfying $$
			\tilde F-F \in H^0(X, \mathcal O(K_X+L)\otimes\mathcal I(\phi+G))$$
		and
		\begin{equation}\label{eq:A2}
		||\tilde F||^2_\phi  \leq \frac{\lambda}{\lambda-1} \limsup_{s\to -\infty} e^{-s}\int_{G<s} i^{n^2} F\wedge \bar F \,e^{-\phi}.
	\end{equation}}
	
	\begin{proof} By Theorem A with $\alpha=1$, it suffices to show that 
		\begin{equation}\label{eq:calculus}
			\limsup_{t\to -\infty} e^{-t}||F||_t^2 \leq \frac{\lambda}{\lambda-1} \limsup_{s\to -\infty} e^{-s}\int_{G<s} i^{n^2} F\wedge \bar F \,e^{-\phi}.
		\end{equation}
		Put
		$
		d\mu:= i^{n^2} F \wedge \bar F \, e^{-\phi}.
		$
		Then 
		$$
		\mu(X)= ||F||^2_\phi, \ \ \ \mu(G<s)=\int_{G<s}  i^{n^2} F \wedge \bar F \, e^{-\phi}, 
		$$
		and we have
		$$
		||F||_t^2= \int_X e^{-\lambda \max\{G-t,0\}} \, d\mu= e^{\lambda t} \mu(X)- \int_{-\infty}^0  \mu(G<s) \,d \,e^{-\lambda \max\{s-t,0\}}.
		$$
		Since $\mu(X)$ is finite and $\lambda>1$, we have
		$$
		\lim_{t\to -\infty} e^{-t} e^{\lambda t} \mu(X) =0,
		$$ and
		hence
		\begin{align*}
			\limsup_{t\to -\infty} e^{-t}||F||_t^2 & = \limsup_{t\to -\infty} \left(- e^{-t} \int_{-\infty}^0  \mu(G<s) \,d \,e^{-\lambda \max\{s-t,0\}}\right)\\
			& =  \limsup_{t\to -\infty} \left(- e^{-t} \int_{t}^0  \mu(G<s) \,d \,e^{-\lambda(s-t)}\right) \\ 
			& =  \limsup_{t\to -\infty} \left(\lambda \int_{t}^0  \mu(G<s) e^{-s} e^{-(\lambda-1)(s-t)}\, ds\right) \\
			& =  \limsup_{t\to -\infty} \left(\lambda \int_{0}^{-t}  \mu(G<x+t) e^{-(x+t)} e^{-(\lambda-1)x}\, dx\right) \\
			&  \leq   \left(\lambda \int_{0}^{\infty} e^{-(\lambda-1)x}\, dx\right) \limsup_{s\to -\infty} \left(\mu(G<s) e^{-s}\right) \\
			&  =  \frac{\lambda}{\lambda-1} \limsup_{s\to -\infty} e^{-s}\int_{G<s} i^{n^2} F\wedge \bar F \,e^{-\phi},
		\end{align*}
		which gives the inequality.
	\end{proof}

	\noindent
	\textbf{Remark.} \emph{In the case that $X$ is a pseudoconvex domain in $\mathbb C^n$ and $G$ is plurisubharmonic, letting $\lambda \to \infty$, Theorem C reduces to the main theorem of Berndtsson and Lempert in \cite{BL}. See also \cite{Al, NW} for other applications of the Berndtsson-Lempert method.}

	\medskip
	
	Let us denote by $Z$ the non-integrable locus of $e^{-\phi-G}$. Then theorem C tells us that $\tilde F$ is a good $L^2$-extension of the restriction to $Z$, say $f:=F|_Z$, of $F$. Of course, $F$ is automatically an $L^2$-extension of $f$. Thus our theorem is only the "estimate" part of the sharp $L^2$-extension theorem. In the case that $X$ is Stein or projective, Ohsawa-Takegoshi \cite{OT} found that the "estimate" part also gives the ''existence'' part  of the $L^2$-extension. For a more general $X$, one may use a partition of unity, $1= \sum \chi_j$, to glue the local $L^2$-extensions $f_j$ and obtain a global smooth $L^2$-extension
	$$
	F:=\sum \chi_j f_j.
	$$
In general, $\dbar F$ is not zero, but its coefficients are always locally square integrable with respect to $e^{-\phi-G}$. We will write for this $\dbar F\in L^2_{loc,\phi+G}$ and use a similar meaning for the notation $L^2_{loc, \phi}$. In the case that $X$ is weakly pseudoconvex K\"ahler, the following sharp Ohsawa-Takegoshi theorem is proved in \cite{ZZ}:

	\begin{theorem}\label{th:ZZ1} Let $(L, e^{-\phi})$ be a holomorphic line bundle with quasi-plurisubharmonic metric (that is, $\phi$ is locally the sum of a smooth function and a plurisubharmonic function) over an $n$-dimensional weakly
	pseudoconvex K\"ahler manifold $X$, and let $G\leq 0$ be a function on $X$ such that
	$$
	\text{$\phi+G$ and $\phi+\lambda G$ are plurisubharmonic} 
	$$
	for some constant $\lambda >1$ on $X$. Then for every smooth section $F$ of $K_X+L$ over $X$ with $F\in L^2_{loc,\phi}$, $\dbar F\in L^2_{loc,\phi+G}$ and 
	\begin{equation}\label{eq:ZZ11}
		||F||_G^2:=\sup_{K \ \text{compact in} \ X}\limsup_{t\to -\infty} \int_{\{t<G<t+1\} \cap K} i^{n^2} F\wedge \overline{F} \,e^{-(\phi+G)}<\infty,
	\end{equation}
	there exists $\tilde F \in  H^0(X, \mathcal O(K_X+L) \otimes \mathcal I(\phi))$ such that $\tilde F - F \in L^2_{loc,\phi+G}$ and
	\begin{equation}\label{eq:ZZ12}
			||\tilde F||^2_\phi   \leq \frac{\lambda}{\lambda-1} ||F||^2_G.
	\end{equation}
\end{theorem}

\noindent
\textbf{Remark.} \emph{As above, one may think of $F$ as a smooth extension of some $f\in H^0(X, \mathcal O(K_X+L) \otimes \mathcal I(\phi)/\mathcal I(\phi+G))$. Then \eqref{eq:ZZ11} is known as the Ohsawa norm (see \cite{Ohsawa}) for $f$. It does not depend the choice of the smooth extension $F$. In fact,  if $F_1, F_2$ are two such smooth extensions (from different partitions of unity and different local $L^2$-extensions) then one may verify that
	$
	F_1-F_2 \in L^2_{loc,\phi+G}.
	$
	Hence we have
	\begin{align}\label{eq:ZZ13}
		\nonumber & \limsup_{t\to -\infty} \int\limits_{\{t<G<t+1\} \cap K} i^{n^2} F_1\wedge \overline{F_1} \, e^{-(\phi+G)} = \limsup_{t\to -\infty} \int\limits_{\{t<G<t+1\} \cap K} i^{n^2} F_2\wedge \overline{F_2} \, e^{-(\phi+G)} 
	\end{align}
	for every compact set $K$ in $X$, from which we get that
	$
	||F_1||^2_G = ||F_2||^2_G.
	$
	Thus Theorem \ref{th:ZZ1} is an equivalent form of  \cite[Theorem 1.2, $R(t)=e^{-t}$]{ZZ}. 
Assume that
\begin{equation}\label{eq:BLOT11}
	\limsup_{s\to -\infty} e^{-s}\int_{\{G<s\}\cap K} i^{n^2} F\wedge \overline{F} \,  e^{-\phi} = \liminf_{s\to -\infty} e^{-s}\int_{\{G<s\}\cap K} i^{n^2} F\wedge \overline{F} \,  e^{-\phi}
\end{equation}
for every compact set $K$ in $X$. Then we have
$$
||F||_G^2  = \sup_{K \ \text{compact in} \ X} \limsup_{s\to -\infty} e^{-s}\int_{\{G<s\}\cap K} i^{n^2} F\wedge \overline{F} \,  e^{-\phi},
$$
and hence 
\begin{equation}\label{eq:BLOT12}
	||F||_G^2   \leq \limsup_{s\to -\infty} e^{-s}\int_{G<s} i^{n^2} F\wedge \overline{F} \,  e^{-\phi}.
\end{equation}}

Thus if \eqref{eq:BLOT11} in the above remark holds and $X$ is weakly pseudoconvex, then Theorem \ref{th:ZZ1} implies our Theorem C (with weaker curvature assumptions). The proofs, however, of the two theorems even in this special case, are quite different. The qck case of Theorem \ref{th:ZZ1} is open as far as we know.\\

We end this section with the following remark:\\

\noindent \textbf{Remark.}
\emph{Using the language of the above remarks, the following sharp Ohsawa-Takegoshi extension theorem may be viewed as a corollary of our Theorem C; note the existence assumption:}

\begin{corollary}
Under the same assumptions as in Theorem C, for any smooth section $F$ of $K_{X}+L$ with $F\in L^{2}_{loc, \phi}, \dbar F\in L^2_{loc, \phi+G}$, if there exists an extension $F_1\in H^0(X,\mathcal{O}(K_X+L)\otimes \mathcal{I}(\phi))$ in the sense that $F_1-F\in L^2_{loc, \phi+G}$, with $||F_1||_{\phi}<\infty$, then we may choose an extension $F_2$ such that 
\begin{align*}
||F_2||_{\phi}^2\leq \frac{\lambda}{\lambda-1}\limsup_{s\to -\infty}e^{-s}\int_{G<s}i^{n^2}F_1\wedge \bar{F}_1 e^{-\phi}.
\end{align*}
\end{corollary}
\begin{proof}
Replace the $F$ that appears in Theorem C by $F_1$. The assertion now follows by writing $F_2-F=(F_2-F_1)-(F-F_1)$ and noting that both of the terms in the parenthesis on the right-hand side lie in $L^2_{loc,\phi+G}$ by assumption. Hence, so does the left-hand side.
\end{proof}

 \section{Hilbert bundle approach to complex Brunn-Minkowski theory and proof of the main theorem}
In this section, we prove the main theorem (Theorem A), assuming the statement of the Jet selection theorem (\ref{th:L21}) (the proof of which comes in section 4). Aside from the latter, the main ingredient in the proof is a version of Berndtsson's convexity theorem in the general setting of qck manifolds (see Theorem \ref{th:main1-general} below). That is, is the convexity of (the logarithm of) certain norms of continuous linear functionals on certain ''Bergman type of spaces'' (to be elaborated below). These norms depend on some (real) parameter $t$, and the idea is to compute the second order derivative of the norms with respect to $t$ directly (and then estimate using $L^2$-methods). For this, we (first) need to know that the norms depend smoothly on $t$. In the classical approach to the complex Brunn-Minkowski theory, in the setting of pseudoconvex domains, this smoothness property, which follows from regularity of the Bergman projection, is established by means of the Hamilton-Kohn regularity theory for the $\dbar$-Neumann operator on smoothly bounded strictly pseudoconvex domains. To deal with the smoothness of the norms in the more general setting, we introduce (what we call) a Hilbert bundle approach to the complex Brunn-Minkowski theory, which allows us to deal with the question of regularity using much more elementary means. In fact, the key ingredient becomes the implicit mapping theorem (see Theorem \ref{th:Bo-1} and its proof below). The argument involving the implicit mapping theorem was taught to us by Berndtsson. To introduce the smoothness result, let us first start by properly defining the corresponding Bergman type of spaces, and norms, that we are going to work with.
	
	\subsection{Positively curved Hilbert bundles}
	
	Let $H$ be a complex Hilbert space. We shall call the product
	$$
	\mathcal H:= H\times \mathbb C
	$$
	a (trivial) Hilbert bundle. By a sub-bundle of $\mathcal{H}$ we will simply mean a Hilbert bundle $\mathcal{H}_0=H_0\times \mathbb{C}$ where $H_0$ is a closed subspace of $H$. Our Bergman type of spaces will be such sub-bundles (of an ambient Hilbert bundle whose underlying Hilbert space is a particular $L^2$-space), and the norms that we shall be working with will be induced from what we call smooth hermitian metrics on Hilbert bundles as $\mathcal{H}$ above; see Definitions \ref{de:smooth hermitian metric} and \ref{de:HH0} below.\\
	
	Now let $\mathcal{H}$ be as above.  We follow \cite[section 2]{Le3} and use the following definition for the notion of holomorphic sections of $\mathcal{H}$:
	
	\begin{definition}\label{de:h-section} Let $U$ be a domain in $\mathbb C$. A smooth mapping  $U \to H$ is said to be holomorphic if its differential is $\mathbb C$-linear. 
	\end{definition}
We then identify (local) holomorphic sections of $\mathcal H$ (over $U$) with holomorphic mappings $U\to H$ and denote by $H^0(U, \mathcal O(\mathcal H))$ the space of holomorphic sections of $\mathcal H$ over $U$; (local) smooth sections of $\mathcal{H}$ on $U$ are of course identified with smooth maps $U\to H$.	We may similarly define holomorphic sections for the dual bundle
		$$
		\mathcal H^*:= H^*\times \mathbb C,
		$$
		where $H^*$ denotes the dual of $H$. Then $\sigma\in H^0(U, \mathcal O(\mathcal H^*))$ if and only if the mapping $U\times H \to \mathbb C$ defined by
		$$
		(\tau, u) \to \sigma(\tau)(u)
		$$
		is holomorphic and $\mathbb C$-linear in the second argument. In particular, each $f\in H^*$ defines a (constant) holomorphic section with the associated mapping $U\times H \to \mathbb C$ given by
		$$
		(\tau, u) \to f(u).
		$$

\begin{definition}\label{de:smooth hermitian metric}	
Let $\mathcal{H}=H\times \mathbb{C}$ be a Hilbert bundle as above. A smooth hermitian metric on $\mathcal{H}$ is a smooth assignment $h:=\{h_{\tau}\}_{\tau\in \mathbb{C}}$ to each $\tau\in \mathbb{C}$ a complete inner product $h_{\tau}$ on $H$, where the smoothness means that for every pair of (local) smooth sections $u$ and $v$ of $\mathcal{H}$, the complex-valued function
\begin{align*}
\tau\mapsto h_{\tau}(u(\tau), v(\tau)):=h(u,v)(\tau)
\end{align*}is smooth.
\end{definition}
Let $L(H)$ denote the space of bounded self-adjoint linear operators $T$ on $H$ such that $T$ is bijective with positive spectrum, and let $h$ be a smooth mapping $\mathbb{C}\to L(H)$. We identify $h$ with the smooth (hermitian) metric on $\mathcal H$ (slightly abusing the notation) given by
		\begin{align*}
		h(u,v)(\tau)&:=(h(\tau)u,v),\quad u,v\in H, \tau\in \mathbb{C},
		\end{align*}where $(\cdot,\cdot)$ denotes the inner product on $H$ (and an element $x\in H$ is viewed as a constant section $\tau \mapsto x$ of $\mathcal{H}$).\\

	\noindent
	\textbf{Remark.} \emph{Given a smooth metric $h:\mathbb{C}\to L(H)$ on $\mathcal{H}$ as above, this defines a family of Hilbert norms $||\cdot||_{h,\tau}:=||\cdot||_{\tau}$ on $H$ where
		$$
		||u||^2_\tau:=h(u,u)(\tau)=(h(\tau) u, u),\quad u\in H.
		$$ By the open mapping theorem, each new norm $||\cdot||_\tau$ is equivalent to the original norm on $H$.}\\
		
		 If $h$ is a smooth hermitian metric on $\mathcal{H}$, we denote the corresponding dual metric on the dual bundle $\mathcal{H}^*$ by $h^*$. The corresponding induced (family of) norm(s) $||\cdot||_{h^*,\tau}:=||\cdot||_{\tau}$ is then given by 
		 \begin{align*}
		 ||f||_{\tau}&=\sup\left\{\abs{f(u)}: ||u||_{\tau}=1\right\},\quad f\in H^*.
		 \end{align*}
		
		We now have the following definition for the the positivity of (the curvature of) smooth metrics $h:\mathbb{C}\to L(H)$ on $\mathcal{H}$:

	\begin{definition}\label{de:h-curvature} Let $h:\mathbb{C}\to L(H)$ be a smooth metric on $\mathcal H$. Then we say that $h$ is positively curved if the dual metric $h^*$ on $\mathcal H^*$ is negatively curved in the following sense: For every open set $U\subset \mathbb C$ and every $\sigma\in H^0(U, \mathcal O(\mathcal H^*)) $, the norm function
		$$
		\tau \mapsto ||\sigma(\tau)||^2_\tau
		$$
		is subharmonic on $U$, where $||\cdot||_{\tau}=||\cdot||_{h^*,\tau}$.
	\end{definition}
	As explained by Berndtsson in page 2362 in  \cite{Bern20}, the condition in the previous definition is equivalent to that
		$\tau\mapsto \log ||\sigma(\tau)||_\tau$ is subharmonic (it can be identically equal to $-\infty$), and we shall usually use the latter directly for the positivity of smooth metrics on $\mathcal{H}$.

	\subsection{Complex Brunn-Minkowski theory}
	
	The  complex Brunn-Minkowski theory established by Berndtsson (see \cite{Bern06, Bern09}, see also \cite{Bern18} for a short survey) is around the positive curvature properties of certain vector bundles associated to holomorphic fibrations. In terms of our Hilbert bundle language, this corresponds to the positive curvature properties of the sub-bundle $\mathcal H_0$ of a fixed ambient Hilbert bundle $\mathcal{H}_1$. The underlying Hilbert spaces of these bundles, respectively $H_0$ and $H_1$, are defined as follows:
	
	\begin{definition}\label{de:HH0} Let $(L, e^{-\phi})$ be a pseudoeffective line bundle over a complex manifold $X$.  We define $H_1$ to be the space of measurable sections $u$ of $K_X+L$ with finite $L^2$-norm (squared)
		$$
		||u||_{\phi}^2:=\int_X i^{n^2} u\wedge \bar u \, e^{-\phi} <\infty,
		$$
		and $H_0$ to be the subspace
		$
		H_0:=H^0(X, \mathcal O(K_X+L))\cap H_1
		$, called
		the Bergman space. 
	\end{definition}
	
	\noindent \textbf{Remark.} \emph{Sometimes we will write $||u||_{\phi}$ simply as $||u||$.}\\

	To define a smooth Hermitian metric on our Hilbert bundles $\mathcal{H}_1$ and $\mathcal{H}_0$, we fix a non-positive measurable function
	$G: X \to [-\infty, 0]$, and a smooth function $\chi:\mathbb{R}\to [0,\infty)$ that vanishes identically on $(-\infty, 0]$. The following proposition is proved in \cite[page 83-85]{Nguyen}.
	
	\begin{proposition}\label{pr:smooth-N} The mapping $h: \mathbb C\to L(H_1)$ defined by
		$$
		h(\tau)(u):= e^{-\chi(G-{\rm Re}\,\tau)} u, \ \ \tau\in \mathbb C, \ \ u\in H_1,
		$$
		is smooth.
	\end{proposition}
	
	We take the induced smooth Hermitian metric from the map $h$ in the previous proposition as a smooth metric on our Hilbert bundle $\mathcal{H}_1=H_1\times \mathbb{C}$. The associated family of new norms $||\cdot||_\tau$, on $H_1$ and $H_0$ is given by
	\begin{equation}\label{eq:tnorm-new}
		||u||^2_\tau:= \int_X i^{n^2} u\wedge \bar u \, e^{-\phi-\chi(G-{\rm Re}\,\tau)}.
	\end{equation} 
	Thus our Hilbert bundles $\mathcal H_1$ and $\mathcal H_0$ are trivial families of vector spaces with non-trivial families of equivalent Hilbert norms. We are interested in the positivity of the curvature of the sub-bundle $\mathcal{H}_0$. Indeed, by Definition \ref{de:h-curvature}, this amounts to the convexity of the map $t \mapsto \log ||f||_{t}$ for every holomorphic section $f$ of the dual of $\mathcal{H}_{0}$. The following definition is of great relevance:
	
	\begin{definition}\label{de:bp} Let $H_1$ and $H_0$ be as in Definition \ref{de:HH0}. We call the canonical orthogonal projection
		$$
		P^\tau: H_1 \to H_0
		$$
		with respect to the above $||\cdot||_\tau$-norm a Bergman projection.
	\end{definition}
	
	The first step in establishing the positivity of the curvature of the sub-bundle $\mathcal{H}_0$, is to establish that $\tau\mapsto ||f(\tau)||^2_{\tau}$ is smooth for every smooth section $f$ of the dual of $\mathcal{H}_0$. To do this, we may prove that $P^\tau$ depends smoothly on $\tau$. The Bergman projection $P^{\tau}$ naturally appears as we are now dealing with continuous linear functionals on the \emph{sub}space $H_0$. This has to do with the fact that the ((1,0)-part of) the Chern connection of $h$ restricted to $\mathcal{H}_0$ is related to that of $h$ on $\mathcal{H}_1$ via the projection map $P^{\tau}$, and that derivatives of the form $\frac{d }{d \tau}||u||^2_{\tau}, u\in H_{i}, i\in \{0,1\}$, can be computed from the ((1,0)-part of) Chern connection of $h$ on $\mathcal{H}_{i}$. We omit here the details. They can be read in \cite[chapter 1]{Nguyen}. For smoothly bounded strongly pseudoconvex domains, the projection map $P^{\tau}$ amounts to the usual (weighted) Bergman projection and its smoothness follows as mentioned earlier from Kohn-Hamilton's regularity theory (see Lemma 2.1 in \cite{Bern06} and Theorem 1.1 in \cite{Bern09}). In our current setting, we have the following regularity theorem mentioned above using an implicit mapping theorem argument (see also \cite{Nguyen} for generalizations):
	
	\begin{theorem}\label{th:Bo-1} With the notation above,
		\begin{eqnarray*}
			P: \mathbb C \times H_1 & \to & H_0 \\
			(\tau,u) & \mapsto & P^\tau u,
		\end{eqnarray*}
		is a smooth map from $\mathbb C \times (H_1, ||\cdot||)$ to $(H_0, ||\cdot||)$.
	\end{theorem}
	
	\begin{proof}  We give a short proof. Fix $(\tau_0, u_0)\in \mathbb R\times H$, and consider the smooth map
		\begin{eqnarray*}
			F: \mathbb C \times H_1\times H_0 & \to & H_0 \\
			(\tau,u,v) & \mapsto & P^{t_0} (  e^{\chi(G-{\rm Re}\,\tau_0)-\chi(G-{\rm Re}\,\tau)} (v-u)  ) 
		\end{eqnarray*}
		(smoothness of $F$ follows from Proposition \ref{pr:smooth-N}). One may show that $F(\tau,u,v)=0$ if and only if  $v=P^\tau u=P(\tau,u)$. Hence, since $F_{H_0}(\tau_0, u_0, v) \equiv id_{H_0}$ (that is, the partial derivative of $F$ in the $H_0$-direction, at $(t_0, u_0, v)$, is equal to the identity map on $H_0$), the smoothness of $P$ follows directly from the implicit mapping theorem.
	\end{proof}
	
	Asssume that $X$ is smoothly bounded strongly pseudoconvex. In \cite[Theorem 1.1]{Bern06}, it is proved by Berndtsson, using H\"ormander's theorem on $L^2$-estimates of solutions of the $\dbar$-operator, that if the total weight is (strictly) plurisubharmonic then $\mathcal H_0$ is positively curved. In our case, the total weight function is
	$$
	\Phi: (z, \tau) \mapsto \phi(z)+\chi(G(z)-{\rm Re}\,\tau).
	$$
	Note that if $\chi''\geq 0$ then
	$$
	i\partial\dbar\Phi \geq  i\partial\dbar \phi+\chi'\, i\partial\dbar G.
	$$
	Thus if $\phi$ and $\phi+\lambda G$ are plurisubharmonic for some constant $\lambda>0$, and $\chi$ is convex with $0\leq \chi' \leq \lambda$, then $\Phi$ is plurisubharmonic. With these assumptions, using Theorem \ref{th:Bo-1} (and $L^2$-methods on qck manifolds; see \cite[chapter 2]{Nguyen}), the first named author gives in \cite{Nguyen} the following version of  \cite[Theorem 1.1]{Bern06} for all qck $X$.
	
	\begin{theorem}[See Theorem 1.6 in page 59 of \cite{Nguyen} ]\label{th:Tai2} Let $(L, e^{-\phi})$ be a pseudoeffective line bundle over an $n$-dimensional qck manifold $X$. Let $G\leq 0$ be a function on $X$ such that
		$\phi+\lambda G$ is plurisubharmonic for some constant $\lambda>0$. Let $\chi \geq 0$ be a smooth convex function on $\mathbb R$ that vanishes identically on $(-\infty, 0)$. Assume that $0\leq \chi' \leq \lambda$. Then the metric defined in \eqref{eq:tnorm-new} on $\mathcal H_0$, by
		
\begin{align*}
||u||^2_\tau:= \int_X i^{n^2} u\wedge \bar u \, e^{-\phi-\chi(G-{\rm Re}\,\tau)},
\end{align*}		
is positively curved in the sense of Definition \ref{de:h-curvature}.
	\end{theorem}

	Any fixed bounded $\mathbb{C}$-linear functional $f: H_0 \to \mathbb C$ defines a constant, thus holomorphic, section of $\mathcal H_0^*$. Hence, the above theorem implies that
		$$
		\tau \mapsto \log ||f||_\tau,
		$$
		is subharmonic on $\mathbb C$, where
		$$
		||f||_\tau:=\sup\{ |f(u)|:  u \in H_0, \ ||u||_\tau=1\}.
		$$
		Since $||f||_\tau$ depends only on the real part of $\tau$, we get that $\log  ||f||_t$ is convex in $t\in\mathbb R$. Letting $\chi(s)$ converge to $\lambda\max\{s,0\}$, we finally obtain the following Berndtsson's convexity theorem in our current setting:
	
	\begin{theorem}\label{th:main1-general} Let $(L, e^{-\phi})$ be a pseudoeffective line bundle over an $n$-dimensional qck manifold $X$. Let $G\leq 0$ be a function on $X$ such that
		$\phi+\lambda G$ is plurisubharmonic for some constant $\lambda>0$. Let $f$ be a bounded $\mathbb{C}$-linear functional on 
		$$
		H_0:= \left\lbrace u\in H^0(X, \mathcal O(K_X+L)) :  \int_X i^{n^2} u\wedge \bar u \, e^{-\phi}<\infty \right\rbrace.
		$$
		Put
		$$
		||f||_t:=\sup \left\lbrace |f(u)|: u\in H_0, \ \int_X i^{n^2} u\wedge \bar u \, e^{-\phi -\lambda \max\{G-t, 0\}}=1 \right\rbrace, \ \ t \in\mathbb R.
		$$
		Then $t\mapsto \log ||f||_t$ is convex on $\mathbb R$. 
	\end{theorem}With this, we are now ready for the proof of the main theorem (Theorem A).
	
	\subsection{Proof of our main theorem}
Assuming the (statement of the) Jet selection theorem (Theorem \ref{th:L21}), we shall follow \cite{Bern20} and prove our main theorem (Theorem A in section 2.2).  Following \cite{Bern20}. We look at quotient spaces of the form
	$$
	Q:=H_0/S.
	$$
	Since $\mathcal{H}_0$ is positively curved (with respect to some smooth hermitian metric $h$), the corresponding Hilbert bundle $\mathcal Q$ is positively curved (this follows from a quotient bundle curvature formula, see \cite[page 43-44]{Nguyen}). Thus Theorem \ref{th:main1-general} applies and gives the following:
	
	\begin{theorem}\label{th:main1-general-Q} Let $(L, e^{-\phi})$ be a pseudoeffective line bundle over an $n$-dimensional qck manifold $X$. Let $G\leq 0$ be a function on $X$ such that
		$\phi+\lambda G$ is plurisubharmonic for some constant $\lambda>0$. Let $f$ be a bounded $\mathbb{C}$-linear functional on the quotient Hilbert space $Q:=H_0/S$. 
		Then $t\mapsto \log ||f||_t$ is convex on $\mathbb R$, where
		$$
		||f||_t:=\sup \left\lbrace |f([u])|: [u]\in H_0/S, \ ||[u]||_t=1 \right\rbrace, \ \ t \in\mathbb R,
		$$ 
		and the quotient norm $||[u]||_t$ on $H_0/S$ (of $[u]\in H_0/S$) is defined by 
		\begin{equation}\label{eq:q-norm}
			||[u]||^2_t:=\inf\left\lbrace \int_X i^{n^2} v\wedge \bar v \, e^{-\phi-\lambda \max\{G-t, 0\} } :   v-u \in S \right\rbrace.
		\end{equation}
	\end{theorem}

	\noindent
	\textbf{Remark.} \emph{In order to prove our main theorem (Theorem A), we apply the above theorem to $S=S_\alpha$, $0\leq \alpha <\lambda$, where
		$$
		S_\alpha:=H_0\cap H^0(X, \mathcal O(K_X+L)\otimes\mathcal I(\phi+\alpha G)).
		$$
		Then one may rephrase our main theorem (Theorem A) as the following:}
	
	\begin{theorem}[Equivalent to the main theorem]\label{th:main-equiv} Let $(L, e^{-\phi})$ be a pseudoeffective line bundle over an $n$-dimensional qck manifold $X$. Let $G\leq 0$ be a function on $X$ such that
		$\phi+\lambda G$ is plurisubharmonic for some constant $\lambda>0$.  Then for every $0\leq \alpha <\lambda$,
		$e^{-\alpha t} ||[u]||_t^2$ is decreasing in $t\in \mathbb R$, where $||[u]||_t$ denotes the quotient norm on $ H_0/S_\alpha$.	
	\end{theorem}
	
	\begin{proof} Fix $t<0$. For every $u\in H_0$, we may take $F\in H_0$ with $[F]=[u]$ (i.e. $F-u\in S_\alpha$) and 
		$
		||F||_t =||[u]||_t.
		$
		Then the Jet selection theorem (Theorem \ref{th:L21}) gives $\tilde F\in H_0$ with  $[\tilde F]=[F]$ and
		\begin{equation}\label{eq:Bo20}
			||[u]||_0^2 \leq  ||\tilde F||_\phi^2 \leq \frac{4\lambda(1-t)^{\lambda -\alpha+1} e^{-\alpha t}}{\lambda-\alpha}\, ||F||_t^2 \leq  C(\varepsilon, \lambda, \alpha) e^{-(\alpha+\varepsilon)t}||[u]||_t^2
		\end{equation}
		for every $\varepsilon>0$. By \eqref{eq:Bo20}, we have
		$$
		||[u]||_0^2 \leq  C(\varepsilon, \lambda, \alpha) e^{-(\alpha+\varepsilon)t}||[u]||_t^2, \ \ \forall \ [u]\in H_0/S_\alpha, \  \ t<0,
		$$
		which is equivalent to the following dual norm estimate:
		$$
		||f||_t^2 \leq C(\varepsilon, \lambda, \alpha) e^{-(\alpha+\varepsilon)t}||f||_0^2, \ \ \forall \ f\in (H_0/S_\alpha)^*, \ \ t<0.
		$$ 
		In particular, we see that  $\log(||f||_t^2)+(\alpha+\varepsilon)t$ is bounded from above at $t=-\infty$. By Theorem \ref{th:main1-general-Q}, $\log(||f||_t^2)+(\alpha+\varepsilon)t$ is convex in $t$. Hence it must be increasing. Letting $\varepsilon$ go to zero, we find that $ \alpha  t + \log (||f||_t^2)$ is increasing for every $f\in (H_0/S_\alpha)^*$. Thus  $e^{-\alpha t } ||[u]||_t^2$ is decreasing for every $[u]\in H_0/S_\alpha$.
	\end{proof}

	\noindent
	\textbf{Remark.} \emph{In the case that $\dim H_0<\infty$, \eqref{eq:Bo20} is precisely \cite[Lemma 2.3]{Bern20}.}

	\section{Proof of the Jet selection theorem in a simple setting}
	
	In this section, we prove the Jet selection theorem (Theorem \ref{th:L21}) in a special case, namely unit ball (local) case. The full proof, which depends heavily on the Demailly approximation theory \cite{D} and the $\dbar$-$L^2$ theory \cite{wang} on qck manifolds, can be found in page 89-91 of \cite{Nguyen}. In the unit ball case, the Jet selection theorem (Theorem \ref{th:L21}) is the following:
	
	\begin{theorem}\label{th:h1} Let $F$ be a holomorphic function on the unit ball $\mathbb B$ in $\mathbb C^n$, $\phi$ be a plurisubharmonic function on $\mathbb B$, and $G\leq 0$ be a function on $\mathbb B$ such that $\phi+ \lambda G$ is  plurisubharmonic
		for some positive constant $\lambda >0$. 
		Then for every $t<0$ and $0\leq \alpha < \lambda$ we can find a holomorphic function $\tilde F $ on $\mathbb B$ such that 
		\begin{equation}\label{eq:h11}
			\int_{\mathbb B} |\tilde F-F|^2 e^{-\phi-\alpha G} \leq  \frac{9 \lambda(1-t)^{\lambda -\alpha+1} e^{-2\alpha t}}{\lambda-\alpha}  ||F||_t^2   + 2 e^{-\alpha t} \int_\mathbb B |F|^2 e^{-\phi}
		\end{equation}
		and
		\begin{equation}\label{eq:h12}
			\int_{\mathbb B} |\tilde F|^2 e^{-\phi} \leq   \frac{4\lambda(1-t)^{\lambda -\alpha+1} e^{-\alpha t}}{\lambda-\alpha} ||F||_t^2,
		\end{equation}
		where
		$
		||F||_t^2:= \displaystyle\int_{\mathbb B} |F|^2 e^{-\phi -\lambda \max\{G-t, 0\} } .
		$
	\end{theorem}
	
	\begin{proof} We shall only prove the case when $G$ and $\phi$ are smooth. For the general case, one may use the convolution with a smoothing kernel to replace $\phi$ and $G$ with $\phi_\varepsilon$ and $G_\varepsilon$. Considering a slightly smaller ball, one may assume that $\phi_\varepsilon$ and  $G_\varepsilon$ are well defined on $\mathbb B$. Since we can choose a rotation invariant smoothing kernel, the submean inequality implies that for every $0\leq \alpha \leq  \lambda$,  $\phi+\alpha G$ is the decreasing limit of $\phi_\varepsilon +\alpha G_\varepsilon$ as $\varepsilon$ goes to zero. In particular, 
		$$
		\phi + \lambda \max\{G-t, 0\} = \max\{\phi+ \lambda G-\lambda t, \phi\}
		$$
		is the decreasing limit of $\phi_\varepsilon + \lambda \max\{G_\varepsilon-t, 0\} $, and thus the general case follows by taking a weak limit of $\tilde F$. In the proof below, one may also replace $\max$ by the regularized $\max$ function. Put
		\begin{equation}\label{eq:sigma-chi1}
			\sigma(s):=\int_{0}^s 1- e^{-\lambda \max\{x-t, 0\} }\, dx.
		\end{equation}
		Then
		$$
		\sigma (0)= 0,  \ \  \sigma'(s)= 1-e^{-\lambda \max\{s-t, 0\}},
		$$
		hence $0\leq \sigma' \leq 1$ and $\sigma'(s)=0$ when $s\leq t$. Thus we have
		\begin{equation}\label{eq:x0}
			t \leq\sigma(t) =\sigma(-\infty) \leq \sigma(G) \leq 0
		\end{equation}
		and
		\begin{equation}\label{eq:x1}
			\sigma''(G)= \lambda  e^{-\lambda (G-t)}  \ \ \text{as $G>t$}.
		\end{equation}
		Put
		$$
		\psi:= -\log(1-\sigma(G)).
		$$
		We then have
		$$
		i\partial\dbar \psi  = \frac{i\partial\dbar \sigma(G)}{1-\sigma(G)} + i\partial \psi \wedge \dbar \psi \geq \frac{\sigma'(G) \, i \partial\dbar G + \sigma''(G) \, i\partial G \wedge \dbar G}{1-\sigma(G)}.
		$$
		Hence our assumption gives
		$$
		i\partial\dbar ( \phi +\alpha G + (\lambda -\alpha) \psi)  \geq  \Theta:=\frac{(\lambda-\alpha)   \sigma''(G) \, i\partial G \wedge \dbar G}{1-\sigma(G)} \geq 0. 
		$$
		Applying the H\"ormander $L^2$-estimate (see \cite[Theorem 6.2]{Bo-dbar}), we get  $$\dbar u = \dbar ((1-\sigma'(G)) F)$$ with
		\begin{equation}\label{eq:x2}
			\int_{\mathbb B} |u|^2 e^{-\phi-\alpha G-  (\lambda -\alpha) \psi } \leq \int_{\mathbb B} |\dbar ((1-\sigma'(G)) F)|^2_{\Theta} e^{-\phi-\alpha G-  (\lambda -\alpha) \psi }.
		\end{equation}
		Since $\sigma'(G) =0$ when $G\leq t$, \eqref{eq:x1} and \eqref{eq:x0} imply 
		\begin{equation}\label{eq:x20}
			|\dbar ((1-\sigma'(G)) F)|^2_{\Theta} \, e^{-\alpha G-  (\lambda -\alpha) \psi } \leq   \frac{\lambda(1-t)^{\lambda -\alpha+1} e^{-\alpha t}}{\lambda-\alpha} \mathbf 1_{G>t} \,e^{-\lambda (G-t)} |F|^2.
		\end{equation}
		Hence \eqref{eq:x2} gives (note that $\psi \leq 0$)
		\begin{equation}\label{eq:x3}
			\int_{\mathbb B} |u|^2 e^{-\phi-\alpha G } \leq  \frac{\lambda(1-t)^{\lambda -\alpha+1} e^{-\alpha t}}{\lambda-\alpha} ||F||_t^2. 
		\end{equation}
		Taking $\tilde F = (1-\sigma'(G)) F -u $, we get
		\begin{equation}\label{eq:x30}
			\int_{\mathbb B} |\tilde F - (1-\sigma'(G)) F|^2e^{-\phi-\alpha G } \leq  \frac{\lambda(1-t)^{\lambda -\alpha+1} e^{-\alpha t}}{\lambda-\alpha}   ||F||_t^2. 
		\end{equation}
		Since $G\leq 0$, the above inequality gives
		\begin{align*}
			\int_\mathbb B |\tilde F|^2 e^{-\phi} & \leq 2 \int_\mathbb B  |\tilde F - (1-\sigma'(G)) F  |^2 e^{-\phi}+2\int_\mathbb B  |(1-\sigma'(G)) F  |^2 e^{-\phi} \\
			& \leq \frac{2\lambda(1-t)^{\lambda -\alpha+1} e^{-\alpha t}}{\lambda-\alpha}   ||F||_t^2  + 2 \int_\mathbb B  |(1-\sigma'(G)) F  |^2 e^{-\phi} \\
			& = \frac{2\lambda(1-t)^{\lambda -\alpha+1} e^{-\alpha t}}{\lambda-\alpha}   ||F||_t^2 + 2  \int_\mathbb B  |F|^2 e^{-\phi -2\lambda \max\{G-t, 0\} }   \\ 
			& \leq  \left(\frac{2\lambda(1-t)^{\lambda -\alpha+1} e^{-\alpha t}}{\lambda-\alpha}  +2\right)  ||F||_t^2 \\
			& \leq  \frac{4\lambda(1-t)^{\lambda -\alpha+1} e^{-\alpha t}}{\lambda-\alpha}  ||F||_t^2.
		\end{align*}
		Thus \eqref{eq:h12} follows. To prove \eqref{eq:h11}, note that \eqref{eq:x30} gives
		\begin{equation}\label{eq:220}
			\int_{G<t} |\tilde F - F|^2e^{-\phi-\alpha G } \leq  \frac{\lambda(1-t)^{\lambda -\alpha+1} e^{-\alpha t}}{\lambda-\alpha}   ||F||_t^2.
		\end{equation}
		Hence together with
		\begin{align*}
			\int_{G\geq t} |\tilde F - F|^2e^{-\phi-\alpha G } & \leq e^{-\alpha t}\int_{G\geq t} |\tilde F - F|^2e^{-\phi} \\
			& \leq 2 e^{-\alpha t} \left(\int_\mathbb B |\tilde F|^2 e^{-\phi} +\int_\mathbb B |F|^2 e^{-\phi} \right),
		\end{align*}
		we see that \eqref{eq:h11} follows from \eqref{eq:220} and \eqref{eq:h12}.
	\end{proof}

	\section{Demailly-Koll\'ar semi-continuity and a Donnelly-Fefferman proof of the strong openness theorem}
	This section is the final section of the paper, and on further applications and results. In here, we shall use Theorem \ref{th:h1} to give a proof of the Demailly-Koll\'ar semi-continuity theorem (see Theorem \ref{th:DK}), and a Donnelly-Fefferman version of Theorem \ref{th:h1} (see Corollary \ref{co:DF}) to give a second proof of Guan-Zhou's strong openness theorem (see Theorem \ref{th:so-DF}).
	
	\subsection{The semi-continuity theorem of Demailly-Koll\'ar}
	
	In \cite{Nad} Nadel obtained a nice criterion for the existence of K\"ahler-Einstein metrics on certain Fano manifolds using his multiplier ideal sheaf. Later, Demailly-Koll\'ar \cite{DK} found a simpler proof of Nadel's criterion using the following semi-continuity result (see \cite[Theorem 0.2]{DK} for the $\phi=0$ case). For the reader's convenience, we shall include a proof here using Theorem \ref{th:h1}, but the idea is already implicitly included in \cite{DK}. 
	
	\begin{theorem}\label{th:DK} Let $\phi$ be a psh function on the unit ball $\mathbb B$ in $\mathbb C^n$. Let $G, G_j\leq 0$ such that
		$$
		\text{$\phi+\lambda G_j$ and $\phi+\lambda G$ are plurisubharmonic}
		$$
		on $\mathbb B$ for some constant $\lambda>1$, and assume that 
		\begin{equation}\label{eq:DK0}
			\int_{\mathbb B} e^{-\phi-\beta G}<\infty, \ \ \lim_{j\to\infty} \int_{\mathbb B} |G_j-G|=0,
		\end{equation}
		for some $1<\beta<\lambda$. Then
		\begin{equation}\label{eq:DK1}
			\lim_{j\to \infty} \int_{|z|<r} |e^{-\phi-G_j}-e^{-\phi-G}|  =0, \  \ \ \forall \ 0<r<1.
		\end{equation}
	\end{theorem}
	In the proof, we use Lemmas \ref{lem:psh} and \ref{lem: lemma below} below.

	\begin{proof}[Proof of Theorem \ref{th:DK}] Assume for contradiction that \eqref{eq:DK1} is not true. Then there exist $0<r_0<1$,  $\varepsilon_0>0$ and a sub-sequence, say $G^{(1)}_{j}$, of $G_j$, with 
		\begin{equation}\label{eq:new}		
			\int_{|z|<r_0} |e^{-\phi-G^{(1)}_j}-e^{-\phi-G}|  \geq \varepsilon_0, \ \ \ \forall  \ j\geq 1.
		\end{equation}
		By Lemma \ref{lem:psh} below, we know that $G^{(1)}_j$ has a sub-sequence $G^{(2)}_j$ which converges to $G$ almost everywhere.	
		We shall use
		\begin{align}\label{eq:DK2}
			\nonumber \int   |e^{-\phi-G^{(2)}_j}-& e^{-\phi-G}|    = \int_{G^{(2)}_j \geq t} |e^{-\phi-G^{(2)}_j}-e^{-\phi-G}|+ \int_{G^{(2)}_j < t} |e^{-\phi-G^{(2)}_j}-e^{-\phi-G}| \\
			& \leq  \int_{G^{(2)}_j \geq t} |e^{-\phi-G^{(2)}_j}-e^{-\phi-G}|+\int_{G^{(2)}_j < t} e^{-\phi-G} + \int_{G^{(2)}_j < t} e^{-\phi-G^{(2)}_j}.
		\end{align}
		The Lebesgue bounded convergence theorem implies that for every fixed $t\leq -1$,
		\begin{equation}\label{eq:DK3}
			\limsup_{j\to \infty}  \left(\int_{G^{(2)}_j \geq t} |e^{-\phi-G^{(2)}_j}-e^{-\phi-G}|+ \int_{G^{(2)}_j < t} e^{-\phi-G} \right)\leq  \int_{G < t+1} e^{-\phi-G}.
		\end{equation}
		The main difficulty is to control the last term $\int_{G^{(2)}_j < t} e^{-\phi-G^{(2)}_j}$ in \eqref{eq:DK2}, and the idea is to use Theorem \ref{th:h1} with $F=1$, $G=G^{(2)}_j$  and $\alpha=1$. Denote $\tilde F$ by $F_{j, t}$. Then \eqref{eq:h12} gives
		\begin{equation}\label{eq:DK4}
			\int_{\mathbb B} |F_{j, t}|^2 e^{-\phi} \leq \frac{4\lambda(1-t)^{\lambda} e^{-t}}{\lambda-1}\, \int_{\mathbb B} e^{-\phi-\lambda\max\{G^{(2)}_j-t, 0\}}
		\end{equation}
		and \eqref{eq:220} gives
		\begin{equation}\label{eq:DK5}
			\int_{G^{(2)}_j <t} |F_{j, t} -1|^2 e^{-\phi-G^{(2)}_j} \leq   \frac{\lambda(1-t)^{\lambda } e^{-t}}{\lambda-1} \int_{\mathbb B} e^{-\phi-\lambda\max\{G^{(2)}_j-t, 0\}}.
		\end{equation}
		By the Lebesgue bounded convergence theorem we have
		\begin{equation}\label{eq:It}
			\lim_{j\to \infty} \int_{\mathbb B} e^{-\phi-\lambda\max\{G^{(2)}_j-t, 0\}}  =  \int_{\mathbb B} e^{-\phi-\lambda\max\{G-t, 0\}}. 
		\end{equation}
		Hence for every fixed $t \leq -1$, we can choose $j(t)$ such that
		$$
		\int_{\mathbb B} e^{-\phi-\lambda\max\{G^{(2)}_j-t, 0\}}  \leq 2  \int_{\mathbb B} e^{-\phi-\lambda\max\{G-t, 0\}},  \ \ \forall \ j\geq j(t).
		$$
		Lemma \ref{lem: lemma below} below further gives
		$$
		e^{-\beta t}\int_{\mathbb B} e^{-\phi-\lambda\max\{G^{(2)}_j-t, 0\}}  \leq \frac{4\lambda}{\lambda-\beta}  \int_{\mathbb B} e^{-\phi-\beta G},  \ \ \forall \ j\geq j(t), \ t\leq -1.
		$$
		Hence \eqref{eq:DK4} implies that there exists $t_0 \leq -1$ such that
		$$
		\sup_{|z|<r} |F_{j, t}(z)| \leq \frac12 , \ \ \forall \ j\geq j(t), \ t\leq t_0.
		$$
		Thus by \eqref{eq:DK5} (when $t_0$ is small enough)
		$$
		\int_{G^{(2)}_j <t, |z|<r}  e^{-\phi-G^{(2)}_j}  \leq e^{(\beta-1)t/2} , \ \ \forall \ j\geq j(t), \ t\leq t_0.
		$$
		By \eqref{eq:DK3}, we can choose $j(t)$ such that  for all $j\geq j(t)$:
		$$
		\int_{G^{(2)}_j \geq t, |z|<r} |e^{-\phi-G^{(2)}_j}-e^{-\phi-G}|+ \int_{G^{(2)}_j < t, |z|<r} e^{-\phi-G} \leq 2 \int_{G < t+1, |z|<r} e^{-\phi-G}.
		$$
		Thus \eqref{eq:DK2} gives
		$$
		\int_{|z|<r}  |e^{-\phi-G^{(2)}_j}-e^{-\phi-G}|   \leq  e^{(\beta-1)t/2}+ 2 \int_{G < t+1, |z|<r} e^{-\phi-G} , \ \ \forall \ j\geq j(t), \ t\leq t_0,
		$$
		which contradicts \eqref{eq:new}. This contradiction shows that \eqref{eq:DK1} must be true.
	\end{proof}

	\begin{lemma}\label{lem:psh} Let $(X, \mu)$ be a measure space, let $f_j$ be a sequence of measurable functions
		on $X$, and suppose that $f_j \to f$ in $L^1$. Then there exists a sub-sequence of $f_j$
		that converges pointwise to $f$ almost everywhere on $X$.
	\end{lemma}
	
	\begin{proof} Let us choose a sub-sequence $f^{(1)}_j$ such that 
		$$
		\sum_{j=1}^\infty \int_X |f^{(1)}_{j+1}-f^{(1)}_j| \, d\mu<\infty.
		$$
		By the monotone convergence theorem, we know that
		$$
		\int_X \sum_{j=1}^\infty |f^{(1)}_{j+1}-f^{(1)}_j| \, d\mu = \sum_{j=1}^\infty \int_X |f^{(1)}_{j+1}-f^{(1)}_j| \, d\mu<\infty.
		$$
		Thus $E:=\left\lbrace \sum_{j=1}^\infty |f^{(1)}_{j+1}-f^{(1)}_j| =\infty \right\rbrace$ is a set of measure zero and $f^{(1)}_j$ converge pointwise to a measurable function $g$ on $X\setminus E$. The dominated convergence theorem further implies that $f^{(1)}_j\to g$ in $L^1$. Thus $g=f$ almost everywhere and $f^{(1)}_j$
		converge pointwise to $f$ almost everywhere on $X$. 
	\end{proof}
	
	\noindent
	\textbf{Remark.} \emph{In general, $L^1$ convergence does not imply almost everywhere convergence, even for plurisubharmonic functions. The following is a counterexample found by the anonymous referee: Take a radial smooth function $0\leq \chi \leq 1$ with compact support in the unit disc and $\chi(w)=1$ for all $|w|<1/2$. Then there exists a positive integer $k$ such that for each $0<r\leq 1/k$, $a\in\mathbb C$,
		$$
		\phi_{a,r}(z):=
		\begin{cases}
			\max\{-1, |z|^2+ r^3 \chi((z-a)/r) \log|z-a| \} &  |z-a|<r; \\
			|z|^2 &  |z-a|\geq r	.
		\end{cases}
		$$
		is subharmonic on $\mathbb C$. Now for any integer $N \geq k$, there exist finite points $a^N_j$, $1 \leq j\leq j_N$ in the unit disc such that these balls $B(a^N_j, e^{-2N^3})$ cover the unit disc. 
		Then the following sequence of subharmonic functions
		\begin{align*}
			\phi(a^k_1 , 1/k), \ \phi(a^k_2 , 1/k), \cdots, & \phi(a^k_{j_k}, 1/k) \\ 
			\phi(a^{k+1}_1 , 1/(k+1)), \ \phi(a^{k+1}_2 , 1/(k+1)), \cdots, & \phi(a^{k+1}_{j_{k+1}}, 1/(k+1)), \cdots \\
			\phi(a^{N}_1 , 1/N), \ \phi(a^N_2 , 1/N), \cdots, & \phi(a^{N}_{j_{N}}, 1/N), \cdots
		\end{align*}
		converges to $|z|^2$ in $L^1$ on a large disc, but does not converge at any point in the unit disc since $\phi(a^{N}_j , 1/N) =-1$ in the ball $B(a^N_j, e^{-2N^3})$.}
	\begin{lemma}\label{lem: lemma below} For every $t\leq -1$ we have 
		$$
		e^{-\beta t}  \int_{\mathbb B} e^{-\phi-\lambda\max\{G-t, 0\}}   \leq \frac{2\lambda}{\lambda-\beta}  \int_{\mathbb B} e^{-\phi-\beta G}.
		$$ 
	\end{lemma}
	
	\begin{proof} Put 
		$$
		I(t):=\int_{\mathbb B} e^{-\phi-\lambda\max\{G-t, 0\}}.		
		$$	
		Notice that the $F=1$ case of \eqref{eq:so4}
		gives
		\begin{equation}\label{eq:so4-new}
			\int_{-\infty}^0  I(t) e^{-\beta t}\, dt   \leq  \frac{\lambda}{(\lambda-\beta)\beta}  \int_{\mathbb B} e^{-\phi-\beta G}.
		\end{equation}
		Since $I(t)$ is increasing in $t$, we have
		$$
		\int_{-\infty}^0  I(s) e^{-\beta s}\, ds  \geq  \int_{t}^{0}  I(s) e^{-\beta s}\, ds \geq I(t) 
		\int_{t}^0  e^{-\beta s}\, ds     \geq  \frac{I(t) e^{-\beta t}}{2\beta}
		$$
		for every $t\leq -1$. Hence the lemma follows from \eqref{eq:so4-new}.
	\end{proof}
	
	\noindent
	\textbf{Remark.} \emph{Thanks to Berndtsson's solution \cite{Bern15b} of the openness conjecture, the following $\beta=1$ case of Theorem \ref{th:DK} also holds:}
	
	\begin{theorem}\label{th:DK-new} Let $\phi$ be a plurisubharmonic function on the unit ball $\mathbb B$ in $\mathbb C^n$, and let $G, G_j\leq 0$ be such that
		$$
		\text{$\phi+\lambda G_j$ and $\phi+\lambda G$ are plurisubharmonic}
		$$
		on $\mathbb B$ for some constant $\lambda>1$. Assume that 
		\begin{equation}\label{eq:DK0-new}
			\int_{\mathbb B} e^{-\phi-G}<\infty, \ \ \lim_{j\to\infty} \int_{\mathbb B} |G_j-G|=0.
		\end{equation}
		Then
		\begin{equation}\label{eq:DK1-new}
			\lim_{j\to \infty} \int_{|z|<r} |e^{-\phi-G_j}-e^{-\phi-G}|  =0, \  \ \ \forall \ 0<r<1.
		\end{equation}
	\end{theorem}
	
	\subsection{A Donnelly-Fefferman proof of the strong openness theorem}
	
	The following Donnelly-Fefferman theorem is a special ($r=1$) case of Theorem 2.1 in \cite{Bern01}:
	
	\begin{theorem}\label{th:DF} Let $\phi$ and $\psi$ be plurisubharmonic functions of class $C^2$ in the unit ball $\mathbb B$ in $\mathbb C^n$, and assume that $\phi$ satisfies
		\begin{equation}\label{eq:DF1}
			i\partial\phi\wedge \dbar\phi \leq i\partial\dbar\phi.
		\end{equation} 
		Then for any smooth $\dbar$-closed $(0,1)$-form $f$ on $\mathbb B$, there exists a solution, $u$, to 
		\begin{equation}\label{eq:DF2}
			\dbar u= f,
		\end{equation} 
		satisfying the estimate  
		\begin{equation}\label{eq:DF3}
			\int_{\mathbb B} |u|^2 e^{-\psi} \leq 4 \int_{\mathbb B} |f|^2_{i\partial\dbar\phi} e^{-\psi}.
		\end{equation}
	\end{theorem}
	
	\begin{proof} We shall follow the proof in \cite{Bern01}. Similar to the proof of Theorem \ref{th:h1}, one may assume that $\phi$ and $\psi$ are $C^2$ up to the boundary and  strictly plurisubharmonic. Let $u$ be the $L^2$-minimal solution of $\dbar (\cdot) =f$ with respect to the weight $\psi+\phi/2$, i.e.
		$$
		\int_{\mathbb B} u \bar h \,e^{-\psi-\phi/2}=0,  \ \forall  \ h \in A^2_{\psi+\phi/2}.
		$$
		Since $\phi$ is assumed to be bounded, we see that $e^{\phi/2} u$ is the $L^2$-minimal solution of the equation $\dbar (\cdot) = \dbar (e^{\phi/2} u)$ with respect to the weight $\phi+\psi$. Thus the H\"ormander estimate gives
		\begin{equation}\label{eq:DF4}
			\int_{\mathbb B} |e^{\phi/2} u|^2 e^{-\phi-\psi} \leq \int_{\mathbb B} |\dbar(e^{\phi/2} u)|_{i\partial\dbar(\phi+\psi)}^2 e^{-\phi-\psi} \leq \int_{\mathbb B} |\dbar(e^{\phi/2} u)|_{i\partial\dbar\phi}^2 e^{-\phi-\psi}.
		\end{equation}
		Note that \eqref{eq:DF1} implies
		$$
		|\dbar(e^{\phi/2} u)|_{i\partial\dbar\phi}^2 e^{-\phi} \leq 2 |f|^2_{i\partial\dbar\phi} +\frac12 |u|^2,
		$$
		and hence \eqref{eq:DF4} gives \eqref{eq:DF3}.
	\end{proof}
	
	The following corollary is a Donnelly-Fefferman version of Theorem \ref{th:h1}:
	
	\begin{corollary}\label{co:DF} Let $F$ be a holomorphic function on the unit ball $\mathbb B$ in $\mathbb C^n$, and let $G 
		\leq  0$ be a plurisubharmonic function on 
		$\mathbb B$. Then for every  $t<-1$ and $\alpha \geq 0$, we can find a holomorphic function $\tilde F $ on $\mathbb B$ such that 
		\begin{equation}\label{eq:h21}
			\int_{\mathbb B} |\tilde F-F|^2 e^{-\alpha G} \leq  24\,|t| \,e^{-2\alpha t}   ||F||_t^2 +2\, e^{-\alpha t}  \int_{\mathbb B} | F|^2
		\end{equation}
		and
		\begin{equation}\label{eq:h22}
			\int_{\mathbb B} |\tilde F|^2 \leq 10\,|t| \, e^{-\alpha t}  ||F||_t^2,
		\end{equation}
		where $||F||_t^2:=\displaystyle \int_{\mathbb B} |F|^2 e^{-  \max\{G-t, 0\}}$.
	\end{corollary}
	
	\begin{proof} Similar to the proof of Theorem \ref{th:h1}, one may assume that $G$ is smooth up to the boundary. Put
		$
		\sigma(s):=\int_{0}^s 1- e^{-\max\{x-t, 0\}} \, dx.
		$
		Applying the above Theorem \ref{th:DF} to
		$$
		\phi:=-\log(-\sigma(G)),  \ \psi:=\alpha G, \ \ v= \dbar ((1-\sigma'(G)) F),
		$$
		we get a solution $u$ of  $\dbar u = \dbar ((1-\sigma'(G)) F)$ with
		$$
		\int_{\mathbb B} |u|^2 e^{-\alpha G} \leq 4 \int_{\mathbb B} |\sigma''(G) \dbar G|^2_{i\partial\dbar \phi }\cdot |F|^2 e^{-\alpha G}.
		$$
		Notice that
		$$
		i\partial\dbar \phi \geq  \frac{i\partial\dbar  (\sigma(G))}{-\sigma(G)} \geq \frac{\sigma''(G) i\partial G \wedge  \dbar G}{ -\sigma(G)}
		$$
		gives
		$$
		|\sigma''(G) \dbar G|^2_{i\partial\dbar \phi } \leq  \sigma''(G) |\sigma(G)| \leq |t| \, e^{-\max\{G-t, 0\}} . 
		$$
		Hence we have
		$$
		\int_{\mathbb B} |u|^2 e^{-\alpha G} \leq  4 |t|  \int_{G\geq t} |F|^2 e^{-\alpha G} e^{-\max\{G-t, 0\}} \leq 4 |t|  e^{-\alpha t} ||F||_t^2.
		$$
		Taking $\tilde F = (1-\sigma'(G)) F -u $, similar to the proof of Theorem \ref{th:h1}, one may verify that $\tilde F$ satisfies \eqref{eq:h21} and \eqref{eq:h22}.
	\end{proof}
	
	We end by showing how to use the above corollary to prove the following Guan-Zhou's strong openness theorem (see \cite{GZ0}):

	\begin{theorem}\label{th:so-DF}
		Let $F$ be a holomorphic function on the unit ball $\mathbb B$ in $\mathbb C^n$, and let $G \leq 0$ be a plurisubharmonic function on 
		$\mathbb B$. If $\displaystyle \int_\mathbb B |F|^2  \,e^{-\beta G}<\infty$ for some $\beta>0$, then there exists $\alpha>\beta$ such that $|F|^2  \,e^{-\alpha G}$ is integrable near the origin of $\mathbb B$.
	\end{theorem} 
	
	\begin{proof}
		
		Assume that $$\alpha_0:=\sup \{\alpha\geq 0: F_0 \in \mathcal I(\alpha G)_0\} <\infty,$$ 
		and
		$$
		I_0:= \inf \left\lbrace\int_{\mathbb B} |\tilde F|^2:  (\tilde F-F)_0 \in \bigcup_{\alpha>\alpha_0}  \mathcal I(\alpha G)_0 \right\rbrace.
		$$
	Step 1 in the proof of Theorem B implies that $I_0>0$. By Corollary \ref{co:DF}, we have
		\begin{equation}\label{eq:h22-new}
			0<I_0 \leq 10 \,|t|\, e^{-\alpha_0 t}   ||F||_t^2 , \ \ \forall \ t<-1.
		\end{equation}
		Since $\displaystyle \int_\mathbb B |F|^2  \,e^{-\beta G}<\infty$, we know that \eqref{eq:so4} (with $\lambda=\beta+1$) gives
		$$
		\int_{-\infty}^0 ||F||_t^2 e^{- \beta t} dt <\infty.
		$$
		By \eqref{eq:h22-new}, we obtain
		$$
		\int_{-\infty}^0 \frac{e^{- \beta t}}{|t|e^{-\alpha_0 t}} dt <\infty,
		$$
		which implies $\alpha_0>\beta$, and completes the proof.
	\end{proof}
	
	\bibliographystyle{amsalpha}

\begin{thebibliography}{A}
		
		\bibitem[Al]{Al} R. Albesiano, {\it A deformation approach to Skoda's Division Theorem}, arXiv:2212.07298.
		
		\bibitem[BG]{BG} S. Bao and Q. Guan, {\it $L^2$
			extension and effectiveness of strong openness property}, Acta
		Mathematica Sinica, English Series, {\bf 38} (2022), 1949--1964.
		
		\bibitem[BG1]{BG1} S. Bao and Q. Guan, {\it $L^2$
			extension and effectiveness of $L^p$ strong openness property}, Acta
		Mathematica Sinica, English Series, {\bf 39} (2023), 814--826.
		
		\bibitem[B1]{B96} B. Berndtsson, {\it The extension theorem of Ohsawa-Takegoshi and the theorem of Donnelly-Fefferman}, Ann. Inst. Fourier (Grenoble) {\bf 46} (1996), 1083--1094.
		
		\bibitem[B2]{B05} B. Berndtsson, {\it Integral formulas and the Ohsawa--Takegoshi extension theorem}, Sci. China
		Ser. A {\bf 48} (2005), suppl., 61--73.
		
		\bibitem[B3]{Bern01} B. Berndtsson, {\it Weighted estimates for the $\dbar$-equation}, Complex analysis and geometry (Columbus,
		OH, 1999), 43--57, Ohio State Univ. Math. Res. Inst. Publ., 9, de Gruyter, Berlin,
		2001.
		
		
		\bibitem[B4]{Bo-dbar} B. Berndtsson, {\it  An Introduction to things $\bar\partial$}, IAS/Park City Math Ser 17, Amer Math Soc, Providence R I , 2010; available in \url{http://www.math.chalmers.se/~bob/7nynot.pdf}. 
		
		\bibitem[B5]{Bern06} B. Berndtsson, {\it Subharmonicity properties of the Bergman kernel and some other functions associated to pseudoconvex domains}, Ann. Inst. Fourier (Grenoble), {\bf 56} (2006), 1633--1662.
		
		\bibitem[B6]{Bern09} B. Berndtsson, {\it Curvature of vector bundles associated to holomorphic fibrations}, Ann. Math. {\bf 169} (2009), 531--560.
		
		\bibitem[B7]{Bern15b} B. Berndtsson, {\it The openness conjecture and complex Brunn-Minkowski inequalities}, Complex geometry
		and dynamics, 29--44, Abel Symp., 10, Springer, Cham, 2015.
		
		\bibitem[B8]{Bern18} B. Berndtsson, {\it
			Complex Brunn-Minkowski theory and positivity of vector bundles}. Proceedings of the International Congress of Mathematicians-Rio de Janeiro 2018, vol. II, Invited Lectures, World Sci. Publ., Hackensack, NJ (2018), 859--884.
		
		\bibitem[B9]{Bern20}  B. Berndtsson, {\it Lelong numbers and vector bundles}, J. Geom. Anal. {\bf 30} (2020),  2361--2376.
		
		\bibitem[BL]{BL} B. Berndtsson and L. Lempert, {\it A proof of the Ohsawa--Takegoshi theorem with sharp estimates}, J. Math. Soc. Japan {\bf 68} (2016), 1461--1472.
		
		\bibitem[BCP]{BCP} B. Berndtsson, J. Cao and M. Paun, {\it On the Ohsawa-Takegoshi extension theorem}, arXiv:2002.04968.
		
		\bibitem[Bl]{Bl0}  Z. Blocki,  {\it Suita conjecture and the Ohsawa--Takegoshi extension theorem}, Invent. Math. {\bf 193}
		(2013), 149--158.
		
		\bibitem[Bl1]{Bl1} Z. Blocki,  {\it Bergman kernel and pluripotential theory, Analysis, complex geometry, and mathematical physics: in honor of Duong H. Phong}, Contemp. Math. {\bf 644}, Amer. Math. Soc., Providence, RI, 2015, 1--10.
		
		\bibitem[CDM]{CDM} J. Cao, J. P. Demailly and S. Matsumura, {\it A general extension theorem for cohomology
			classes on non reduced analytic subspaces}, Sci. China Math. {\bf 60} (2017), 949--962.
		
		\bibitem[Ch]{Chan} T. Chan, {\it On an $L^2$ extension theorem from log-canonical centres with log-canonical
			measures}, arXiv: 2008.03019.
		
		\bibitem[CC]{CC} T. Chan and Y. J. Choi, {\it Extension with log-canonical measures and an
			improvement to the plt extension of Demailly-Hacon-Paun},  Math. Ann. (2021).
		
		\bibitem[C]{Chen} B. Chen, {\it A simple proof of the Ohsawa-Takegoshi extension theorem}, arXiv:1105.2430.
		
		\bibitem[CWW]{CWW} B. Chen, J. Wu and X. Wang, {\it Ohsawa-Takegoshi type theorem and extension of plurisubharmonic functions},
		Math. Ann. {\bf 262} (2015), 305--319.
		
		\bibitem[De]{Demailly12} J.-P. Demailly, {\it Complex analytic and differential geometry}, book in \url{https://www-fourier.ujf-grenoble.fr/~demailly/manuscripts/agbook.pdf}.
		
		
		\bibitem[De1]{Dem00} J. P. Demailly, {\it Multiplier ideal sheaves and analytic methods in algebraic geometry}, School
		on Vanishing Theorems and Effective Results in Algebraic Geometry (Trieste, 2000), 1--148,
		ICTP Lect. Notes, 6, Abdus Salam Int. Cent. Theoret. Phys., Trieste, 2001.
		
		\bibitem[De2]{D} J. P. Demailly, {\it Analytic methods in algebraic geometry}, lecture notes in \url{https://www-fourier.ujf-grenoble.fr/~demailly/manuscripts/analmeth.pdf}.
		
		
		\bibitem[De3]{Dem18} J. P. Demailly, {\it Extension of holomorphic functions and cohomology classes from non reduced analytic subvarieties}, Geometric Complex Analysis, Springer Proceedings in Mathematics and Statistics 246, Springer, Singapore, 2018, 97--113.
		
		\bibitem[DHP]{DHP} J. P. Demailly, C. D. Hacon and M.  Paun, {\it Extension theorems, non-vanishing and the existence of good minimal models}, Acta Mathematica {\bf 210} (2013), 203--259.
		
		\bibitem[DK]{DK} J. P. Demailly and J. Koll\'ar, {\it Semicontinuity of complex singularity exponents and K\"ahler-Einstein metrics on Fano orbifolds}, Ann. Sci. Ecole Norm. Sup. {\bf 34} (2001), 525--556.
		
		\bibitem[FJ]{FJ} C. Favre and M. Jonsson, {\it Valuations and multiplier ideals}, J. Am. Math. Soc. {\bf 18} (2005), 655--684.
		
		
		\bibitem[Gu]{G} Q. Guan, {\it A sharp effectiveness result of Demailly's strong openness conjecture}, Adv. Math. {\bf 348} (2019), 51--80.
		
		\bibitem[Gu1]{G1} Q. Guan, {\it General concavity of minimal $L^2$ integrals related to multiplier sheaves}, arXiv:1811.03261v4.
		
		\bibitem[GM]{GM} Q. A. Guan and Z. T. Mi, {\it Concavity of minimal $L^2$ integrals related to multipler ideal sheaves}, Peking Math J. {\bf 6} (2023), 393--457.
		
		
		\bibitem[GY]{GY} Q. A. Guan and Z. Yuan, {\it Effectiveness of strong openness property in $L^p$}, Proc. Amer. Math. Soc. {\bf 151} (2023), 4331-4339. 
		
		
		
		\bibitem[GZ1]{GZ} Q. Guan and X. Zhou, {\it A solution of an $L^2$ extension problem with an optimal estimate and
			applications}, Ann. of Math. (2), {\bf 181} (2015), 1139--1208.
		
		\bibitem[GZ2]{GZ0}  Q. Guan and X. Zhou, {\it A proof of Demailly's strong openness conjecture}, Ann. Math. {\bf 182} (2015) 605--616.
		
		\bibitem[GZ3]{GZ1} Q. Guan and X. Zhou, {\it Effectiveness of Demailly's strong openness conjecture and related
			problems}, Invent. Math. {\bf 202} (2015), 635--676.
		
		\bibitem[Hi]{Hiep} P. H. Hiep, {\it The weighted log canonical threshold}, C. R. Math. Acad. Sci. Paris {\bf 352} (2014), 283--288.
		
		\bibitem[Ho]{Hosono0} G. Hosono, {\it On sharper estimates of Ohsawa--Takegoshi $L^2$-extension theorem}, J. Math.
		Soc. Japan {\bf 71} (2019), 909--914.
		
		\bibitem[JM]{JM} M. Jonsson and M. Mustata, {\it Valuations and asymptotic invariants for sequences of ideals},
		Annales de l'Institut Fourier {\bf 62} (2012), 2145--2209.
		
		
		\bibitem[K1]{K1} D. Kim {\it $L^2$ extension of adjoint line bundle sections}, Ann. Inst. Fourier (Grenoble) {\bf 60} (2010), 1435--1477.
		
		\bibitem[K2]{K2} D. Kim, {\it A remark on $L^2$ extension theorems}, Bull. Korean Math. Soc. {\bf 58} (2021), 1235--1245. 
		
		\bibitem[KS]{KS} D. Kim and H. Seo, {\it On $L^2$ extension from singular hypersurfaces}, 	arXiv:2104.03554.
		
		
		\bibitem[Le1]{Lempert} L. Lempert, {\it Modules of square integrable holomorphic germs}, Analysis meets geometry, 311--333, Trends Math.,
		Birkhuser/Springer, Cham, 2017.
		
		\bibitem[Le2]{Le2} L. Lempert, {\it Extrapolation, a technique to estimate}, arXiv:1507.06216.
		
		\bibitem[Le3]{Le3} L. Lempert, {\it A maximal principle for Hermitian (and other) metrics}, Proc. Am. Math. Soc. {\bf 143} (2015), 2193--2200.
		
		
		
		
		\bibitem[Ma]{Ma} S. Matsumura, {\it An injectivity theorem with multiplier ideal sheaves of singular metrics with transcendental singularities}, J. Algebraic Geom. {\bf 27} (2018), 305--337.
		
		\bibitem[MV1]{MV1} J. D. McNeal and D. Varolin, {\it Analytic inversion of adjunction: $L^2$ extension theorems with gain}, Ann. Inst. Fourier (Grenoble) {\bf 57} (2007), 703--718.
		
		
		\bibitem[MV2]{MV2}  J. D. McNeal and D. Varolin,{\it $L^2$ Extension of dbar-closed forms from a hypersurface}, 
		J. Anal. Math. {\bf 139} (2019), 421--451.
		
		
		\bibitem[MV3]{MV3}  J. D. McNeal and D. Varolin,{\it Extension of Jets With $L^2$ estimates, and an application}, arXiv:1707.04483.
		
		\bibitem[Na]{Nad} A.M. Nadel, {\it Multiplier ideal sheaves and K\"ahler-Einstein metrics of positive scalar curvature}, Proc. Nat. Acad. Sci. U.S.A., {\bf 86} (1989), 7299--7300 and Annals of Math.,
		{\bf 132} (1990), 549--596.
		
		\bibitem[Ng]{Nguyen} T. Nguyen, {\it A Hilbert bundles description of complex Brunn-Minkowski theory}, arXiv:2305.07476v1.
		
		\bibitem[NW]{NW} T. Nguyen and X. Wang, {\it On a remark by Ohsawa related to the Berndtsson-Lempert method for $L^2$ holomorphic extension}, to appear in Ark. Mat.
		
		\bibitem[O2]{O2} T. Ohsawa, {\it On the extension of $L^2$ holomorphic functions II}, Publ. Res. Inst. Math.
		Sci. {\bf 24} (1988), 265--275.
		
		\bibitem[O3]{O3} T. Ohsawa, {\it On the extension of $L^2$ holomorphic functions III. Negligible weights}, Math. Z. {\bf 219} (1995), 215--225.
		
		
		\bibitem[Oh]{Ohsawa} T. Ohsawa, {\it On the extension of $L^2$-holomorphic functions V, effect of generalization}, Nagoya Math J. {\bf 161}
		(2001), 1--21.
		
		\bibitem[OT]{OT} T. Ohsawa and K. Takegoshi, {\it On the extension of $L^2$-holomorphic functions}, Math. Z. {\bf 195} (1987), 197--204.
		
		\bibitem[Siu]{Siu} Y. T. Siu, {\it Invariance of plurigenera}, Invent. Math. {\bf 134} (1998), 661--673.
		
		
		
		\bibitem[Su]{Suita} N. Suita, {\it Capacities and kernels on Riemann surfaces}, Arch. Rational Mech. Anal. {\bf 46} (1972), 212--217.
		
		
		
		\bibitem[W1]{wang21} X. Wang, {\it An explicit estimate of the Bergman kernel for positive line bundles}, to appear in Ann. Fac. Sci. Toulouse Math.
		
		\bibitem[W2]{wang} X. Wang, {\it H\"ormander $\dbar$ theory on quasi-complete K\"ahler manifold}, lecture notes in \url{https://folk.ntnu.no/xuwa/res/course_all.pdf}.
		
		
		\bibitem[ZZ]{ZZ}  X. Zhou and L. Zhu, {\it Extension of cohomology classes and holomorphic sections defined on subvarieties}, arXiv: 1909.08822, to appear in J. Algebraic Geom.
		
	\end{thebibliography}

\end{document}